\documentclass{amsart}
\usepackage{amsmath, amssymb, verbatim, amscd, amsthm, bm}
\usepackage[normalem]{ulem}
\usepackage[driverfallback=dvipdfm]{hyperref}
\usepackage[dvipdfm]{graphicx}
\usepackage{color,graphicx,shortvrb}
\usepackage{enumerate}
\usepackage{cancel}
\usepackage{dsfont}

\newtheorem{thm}{Theorem}[section]

\newtheorem{lem}{Lemma}[section] 

\newtheorem{cor}[thm]{Corollary}

\theoremstyle{definition}

\newtheorem*{que}{Problem}%[section]
\numberwithin{equation}{section}  
\theoremstyle{remark}
\newtheorem*{remark}{Remark}

\newcommand{\q}{\quad} 
\newcommand{\qq}{\qquad} 
\newcommand{\qbox}[1]{\q \mbox{#1} \q} 
\newcommand{\qqbox}[1]{\qq \mbox{#1} \qq}

\newcommand{\N}{\mathbb N} 
\newcommand{\R}{\mathbb R} 
\newcommand{\set}[1]{\left\{ #1  \right\}}
\newcommand{\fr}[2]{\frac{#1}{#2}}
\newcommand{\Vinf}[1]{\Vert{#1}\Vert_\infty}

\newcommand{\K}{K}

\newcommand{\SCK}{S(C_0(K))^+}
\newcommand{\SCX}{S(C_0(X))^+}
\newcommand{\SCY}{S(C_0(Y))^+}

\newcommand{\CK}{C_0(\K)}
\newcommand{\CX}{C_0(\X)}
\newcommand{\CY}{C_0(\Y)}
\newcommand{\DK}{D_K}
\newcommand{\DX}{D_{\X}}
\newcommand{\DY}{D_{Y}}
\newcommand{\ZK}[1]{Z_\K({#1})}
\newcommand{\ZX}[1]{Z_\X({#1})}
\newcommand{\ZY}[1]{Z_\Y({#1})}
\newcommand{\MK}[1]{M_\K({#1})}
\newcommand{\MX}[1]{M_\X({#1})}
\newcommand{\MY}[1]{M_\Y({#1})}
\newcommand{\PX}[1]{P_\X({#1})}
\newcommand{\PY}[1]{P_\Y({#1})}
\newcommand{\PK}[1]{P_\K({#1})}
\newcommand{\Phii}{\Phi^{-1}}
\newcommand{\Phit}{\tilde{\Phi}}
\newcommand{\tz}{t_0}
\newcommand{\sii}{\si^{-1}}
\newcommand{\si}{\sigma}
\newcommand{\xo}{x_1}

\newcommand{\xz}{x_0}
\newcommand{\yz}{y_0}
\newcommand{\yo}{y_1}

\newcommand{\X}{X}
\newcommand{\Y}{Y}

 \usepackage{marginnote}
 \usepackage{geometry}

\begin{document}

\title[A variant of Tingley's problem
for positive unit spheres]
{A variant of Tingley's problem
for positive unit spheres of
continuous functions that vanish at infinity}

\author[K. Ezumi]{Kazuki Ezumi}
\address[K. Ezumi]
{Department of Mathematics,
Faculty of Science,
Niigata University,
Niigata 950-2181, Japan}
\email{s22s114j@mail.cc.niigata-u.ac.jp}

\author[M. R. Lin]{Min-Ruei Lin}
\address[M. R. Lin]
{Graduate School of Science and Technology, Niigata University, Niigata 950-2181, Japan;
And
Department of Applied Mathematics,
National Sun Yat-sen University,
Kaohsiung, 80424, Taiwan}
\email{m082030021@student.nsysu.edu.tw}

\author[T. Miura]{Takeshi Miura}
\address[T. Miura]
{Department of Mathematics,
Faculty of Science,
Niigata University,
Niigata 950-2181, Japan}
\email{miura@math.sc.niigata-u.ac.jp}

\begin{abstract}
Let $\SCX$ and $\SCY$ denote the positive parts of the unit spheres of $\CX$ and $\CY$,
where $\X$ and $\Y$ are locally compact
Hausdorff spaces. 
We prove that every surjective isometry
from $\SCX$ onto $\SCY$ is a composition
operator induced by a homeomorphism
between $\X$ and $\Y$. 
As a consequence, such a map extends
to a surjective real-linear isometry from
$\CX$ onto $\CY$.
We also characterize surjective
phase-isometries on the positive unit sphere.
\end{abstract}

\maketitle

\section{Introduction}

Let $E$ and $F$ be Banach spaces
with unit spheres $S(E)$ and $S(F)$, respectively.  
In 1987, Tingley \cite{ting} asked whether
every surjective isometry between $S(E)$ and $S(F)$
extends to a surjective isometry between
the entire spaces $E$ and $F$. 
This problem, now known as Tingley's problem,
has attracted considerable attention
in the context of function spaces
(see, for example,
\cite{cue1, cue2, hat1, hat2, hat3,
hiro2, wang1, wang2, wang4}).

When $E$ is an ordered Banach space, we denote
\[
S(E)^+=\{a\in S(E):a\geq 0\},
\]
and call $S(E)^+$ the \textit{positive unit sphere}
of $E$.  
Peralta \cite{pera} proposed a variant of
Tingley's problem for positive unit spheres
of certain Banach spaces.  
Later, Leung, Ng and Wong \cite{leun}
generalized Peralta's question as follows.
\begin{que}
Let $\Phi\colon S(E)^+\to S(F)^+$ be a
surjective isometry between the positive
unit spheres of ordered Banach spaces
$E$ and $F$, that is,
\[
\|\Phi(a)-\Phi(b)\|
=\|a-b\|\qq(a,b\in S(E)^+).
\]
Does $\Phi$ extend to a surjective
real-linear isometry between
the entire spaces $E$ and $F$?
\end{que}
Leung, Ng and Wong \cite{leun} solved
this problem affirmatively for $L^p$ spaces
with $1\leq p\leq \infty$, as well as for
$C(K)$ spaces of all continuous functions
on a compact Hausdorff space $K$, in which the authors assume the algebras are unital.
In the argument for $C(K)$, they use the constant function $\mathds{1}$ to construct strictly positive functions.
However, in $\CX$ with $\X$ locally compact Hausdorff but not compact, the absence of $\mathds{1}$ prevents a direct adaptation of their argument.

In this paper, we prove that the problem also has a positive answer for the non-unital case, that is, when $E=\CX$ and $F=\CY$, the spaces of continuous real-valued functions on locally compact Hausdorff spaces $X$ and $Y$ that vanish at infinity, respectively.
More precisely, we show that any such isometry $\Phi$ is represented as a composition operator induced by a homeomorphism $\sigma$ between the underlying spaces $X, Y$.

Our proof consists of several steps.
The main task is to define a suitable map
$\si$ between the underlying spaces.  
We first establish that inclusions of
maximum sets are preserved by $\Phi$,
and then show that a family of peak sets
has a non-empty intersection.
This property enables us to construct
$\si$ and to verify that $\Phi$ is indeed
a composition operator induced by $\si$.

The paper is organized as follows. 
In Section~\ref{sect2} we collect
basic notation and definitions
that will be used throughout the paper. 
Section~\ref{sect3} states the main theorem
and provides its proof, together with
the necessary lemmas.
Section~\ref{sect4} contains an application
to surjective phase-isometries
on the positive unit sphere.

\section{Preliminaries}\label{sect2}

Let $\CK$ be the real Banach space
of all continuous real-valued functions
on a locally compact Hausdorff
space $K$ that vanish at infinity,
equipped with the supremum norm
$\Vinf{\cdot}$.
We denote
\[
\SCK
=\set{f\in\CK:\Vinf{f}=1, f\geq 0},
\]
the positive part of the unit sphere of $\CK$, and call it \textit{the positive unit sphere of $\CK$}.
For each $f\in\SCK $, we set
\[
\MK{f} :=\set{t\in\K:f(t)=1},\qq
\ZK{f} :=\set{t\in\K:f(t)=0},
\]
which we call the maximum set and zero set
of $f$, respectively.
Following Leung, Ng and Wong~\cite{leun}, we define for any $f\in\SCK$ the set
\[
 \DK(f) :=\set{h\in\SCK:\Vinf{h-f}<1}.
\]
This set will play a key role in our analysis,
and will be studied further
in Lemma~\ref{lem3.1}.
Note that $ \DK(f)\neq\emptyset$,
as $f\in\DK(f)$.

\section{Main results}\label{sect3}

This section states and proves
the main result.

\begin{thm}\label{thm3.1}
Let $\X$ and $\Y$ be
locally compact Hausdorff spaces.
Let $\Phi\colon\SCX\to\SCY$ be a
surjective isometry between the positive unit spheres of $\CX$ and $\CY$.
Then
there exists a homeomorphism
$\si\colon\Y\to\X$ such that
$\Phi(f)=f\circ\si$ for all $f\in\SCX$.
\end{thm}

\begin{cor}\label{cor1}
Let $\Phi\colon\SCX\to\SCY$ be
a surjective isometry.
Then $\Phi$ extends uniquely
to a surjective real-linear isometry
from $\CX$ onto $\CY$.
\end{cor}

As a first step, we state a basic criterion
for membership in $ \DK(f)$,
which will be used repeatedly in what follows.
Although the proof is straightforward,
we give it for the sake of completeness.

\begin{lem}\label{lem3.1}
Let $f,h\in\SCK$.
Then $h\in \DK(f)$ if and only if
the following equalities hold:
\begin{equation}\label{lem3.1.1}
\MK{h}\cap\ZK{f}
=\emptyset
=\ZK{h}\cap\MK{f}.
\end{equation}
\end{lem}

\begin{proof}
Suppose that $h\in \DK(f)$.
Then $\Vinf{h-f}<1$ by definition.
For each $t_0\in\MK{h}$, we have
the following inequalities:
\[
1-f(t_0)
=h(t_0)-f(t_0)
\leq|h(t_0)-f(t_0)|
\leq\Vinf{h-f}<1,
\]
which shows that $f(t_0)>0$.
This yields
$\MK{h}\cap\ZK{f}=\emptyset$.
Similarly, for each $t_1\in\MK{f}$,
we obtain the following result:
\[
1-h(t_1)
=f(t_1)-h(t_1)
\leq|f(t_1)-h(t_1)|
\leq\Vinf{f-h}<1,
\]
which proves that $h(t_1)>0$.
This shows that
$\ZK{h}\cap\MK{f}=\emptyset$.
Thus \eqref{lem3.1.1} holds.

Conversely, we assume that
equalities \eqref{lem3.1.1} hold.
We prove that $\Vinf{h-f}<1$.
Fix an arbitrary $t\in\K$.
Because $f(t), h(t)\in [0,1]$, we obtain the following inequalities:
\begin{align*}
-1\leq-f(t)
\leq h(t)-f(t)
\leq h(t)
\leq1.
\end{align*}
In particular,
$\Vinf{h-f}=\sup_{s\in\K}|h(s)-f(s)|\leq1$.
Suppose, for contradiction, $\Vinf{h-f}=1$.
Then there exists an $s_0\in\K$
with $|h(s_0)-f(s_0)|=1$.
Without loss of generality,
we may and do assume that
$h(s_0)-f(s_0)=1$.
We derive
from the preceding inequalities that
\[
1=h(s_0)-f(s_0)\leq h(s_0)\leq1.
\]
This shows that $h(s_0)=1$
and $f(s_0)=0$.
This implies that
$s_0\in\MK{h}\cap\ZK{f}$,
which is in contradiction with
$\MK{h}\cap\ZK{f}=\emptyset$.
This shows that
$\Vinf{h-f}\neq1$, thus
$\Vinf{h-f}<1$.
As a consequence, we obtain
$h\in \DK(f)$.
This establishes the lemma.
\end{proof}

For each $\tz\in\K$, we define a subset
$\PK{\tz}$ of $\SCK$ as follows:
\[
\PK{\tz} := \set{h\in\SCK:h(\tz)=1}.
\]
We refer to $\PK{\tz}$ as the family
of peak functions each of which peaks at $\tz$.
Although elementary, the following lemma will be used repeatedly in what follows.

\begin{lem}\label{lem3.2}
Let $\tz,t_1\in\K$.
If $\PK{\tz}\subseteq\PK{t_1}$,
then $\tz=t_1$ holds.
\end{lem}

\begin{proof}
Suppose that $\tz\neq t_1$,
then there exists an $f_0\in\SCK$ such that
$f_0(\tz)=1$ and $f_0(t_1)=0$,
and thus
$f_0\in\PK{\tz}\setminus\PK{t_1}$.
Therefore the assumption $\tz\neq t_1$
implies $\PK{\tz}\not\subseteq\PK{t_1}$, completing the proof.
\end{proof}

The following lemma establishes a connection between the zero sets and the maximum sets of
functions and the associated sets $ \DK(f)$ and $ \DK(g)$.

\begin{comment}
The following lemma shows
how $\DY(v)$ and $\DY(u)$ characterize the corresponding maximum sets,
providing a partial converse
of Lemma~\ref{lem3.3}.
\end{comment}
\begin{comment}
\begin{lem}\label{lem3.5}
Let $u,v\in\SCY$.
If $\DY(v)\subseteq\DY(u)$, then $\MY{u}\subseteq\MY{v}$.
\end{lem}

\begin{proof}
Suppose that $\MY{u}\not\subseteq\MY{v}$, we prove that $\DY(v)\not\subseteq\DY(u)$.
There exists a point
$\yz\in\MY{u}\setminus\MY{v}$.
Let $v_0\in\SCY$ be such that
$v_0(\yz)=0$ and $v_0=1$
on $\MY{v}$.
We set $w =v_0v$.
Then $w\in\SCY$
with $\yz\in\ZY{w}$ and $\MY{w}=\MY{v}$.
By the choice of $\yz$, we conclude that
$\yz\in\ZY{w}\cap\MY{u}$.
This yields
$\ZY{w}\cap\MY{u}\neq\emptyset$.
It follows from Lemma~\ref{lem3.1}
that $w\not\in\DY(u)$.
On the other hand,
we have the following identities
by definition:
\[
\MY{v}\cap\ZY{v}=\emptyset
\qqbox{and}
\ZY{w}\cap\MY{w}=\emptyset.
\]
By applying $\MY{w}=\MY{v}$,
we conclude that
\[
\MY{w}\cap\ZY{v}=\emptyset
\qqbox{and}
\ZY{w}\cap\MY{v}=\emptyset.
\]
Now, we apply
Lemma~\ref{lem3.1} to conclude that
$w\in\DY(v)$.
From this we deduce $w\in\DY(v)\setminus\DY(u)$,
which shows $\DY(v)\not\subseteq\DY(u)$.
We conclude that $\DY(v)\subseteq\DY(u)$
implies $\MY{u}\subseteq\MY{v}$.
Therefore the claim follows.
\end{proof}
\end{comment}

\begin{lem}\label{lem3.3}
Let $f,g\in\SCK$.
If $\MK{g}\subseteq\MK{f}$ and $\ZK{g}\subseteq\ZK{f}$, then $ \DK(f)\subseteq \DK(g)$.
Conversely, if $\DK(f)\subseteq\DK(g)$, then $\MK{g}\subseteq\MK{f}$.
\end{lem}

\begin{proof}
Fix an arbitrary $h\in \DK(f)$.
We deduce from Lemma~\ref{lem3.1}
that
\[
\MK{h}\cap\ZK{f}=\emptyset=\ZK{h}\cap\MK{f}.
\]
By the assumptions, we have
the following inclusions:
\begin{align*}
\MK{h}\cap \ZK{g}
&\subseteq
\MK{h}\cap\ZK{f}
=\emptyset,\\
\ZK{h}\cap\MK{g}
&\subseteq
\ZK{h}\cap\MK{f}
=\emptyset.
\end{align*}
This shows that
$\MK{h}\cap\ZK{g}
=\emptyset
=\ZK{h}\cap\MK{g}$,
which implies that $h\in \DK(g)$
by Lemma~\ref{lem3.1}.
Hence $ \DK(f)\subseteq\DK(g)$.

Conversely, suppose that $\MK{g}\not\subseteq\MK{f}$, we prove that $\DK(f)\not\subseteq\DK(g)$.
There exists a point
$\yz\in\MK{g}\setminus\MK{f}$.
Let $v_0\in\SCK$ be such that
$v_0(\yz)=0$ and $v_0=1$
on $\MK{f}$.
We set $w =v_0f$.
Then $w\in\SCK$
with $\yz\in\ZK{w}$ and $\MK{w}=\MK{f}$.
By the choice of $\yz$, we conclude that
$\yz\in\ZK{w}\cap\MK{g}$.
This yields
$\ZK{w}\cap\MK{g}\neq\emptyset$.
It follows from Lemma~\ref{lem3.1}
that $w\not\in\DK(g)$.
On the other hand,
we have the following identities
by definition:
\[
\MK{f}\cap\ZK{f}=\emptyset
\qqbox{and}
\ZK{w}\cap\MK{w}=\emptyset.
\]
By applying $\MK{w}=\MK{f}$,
we conclude that
\[
\MK{w}\cap\ZK{f}=\emptyset
\qqbox{and}
\ZK{w}\cap\MK{f}=\emptyset.
\]
Now, we apply
Lemma~\ref{lem3.1} to conclude that
$w\in\DK(f)$.
From this we deduce $w\in\DK(f)\setminus\DK(g)$,
which shows $\DK(f)\not\subseteq\DK(g)$.
We conclude that $\DK(f)\subseteq\DK(g)$
implies $\MK{g}\subseteq\MK{f}$.
Therefore, the claim follows.
This completes the proof.
\end{proof}

In the remainder of this paper,
we assume that $\Phi$ is a surjective
isometry from $\SCX$ onto $\SCY$,
where $X$ and $Y$ are locally compact
Hausdorff spaces.
Then the inverse map $\Phii\colon\SCY\to\SCX$ is also a well-defined surjective isometry.

At this stage, we clarify the structure
of the image under $\Phi$ by relating
$\Phi(\DX(f))$ and $ \DY(\Phi(f))$.

\begin{lem}\label{lem3.4}
For each $f\in\SCX$, we have
$\Phi(\DX(f))= \DY(\Phi(f))$.
\end{lem}

\begin{proof}
We first show that
$\Phi(\DX(f))\subseteq\DY(\Phi(f))$.
Fix an arbitrary $u\in\Phi(\DX(f))$.
By the surjectivity of $\Phi$, 
there exists $h\in\DX(f)$
such that $u=\Phi(h)$.
Because $\Phi$ is an isometry and $h\in\DX(f)$, we obtain the following inequalities:
\[
\Vinf{u-\Phi(f)}
=\Vinf{\Phi(h)-\Phi(f)}
=\Vinf{h-f}<1.
\]
This shows that
$u\in\DY(\Phi(f))$.
Therefore we conclude that
$\Phi(\DX(f))\subseteq\DY(\Phi(f))$.
To prove the reverse inclusion,
we apply the same argument
to the pair of $(\Phii,\Phi(f))$,
instead of $(\Phi,f)$.
It follows that
$\Phii(\DY(\Phi(f)))
\subseteq\DX(\Phii(\Phi(f)))=\DX(f)$.
This proves that $\DY(\Phi(f))\subseteq\Phi(\DX(f))$.
This completes the proof
of the identity.
\end{proof}

We next prove that finite families of
peak functions yield non-empty
intersections of their corresponding maximum sets.
This is the first step
toward the general case.

\begin{lem}\label{lem3.6}
Let $\xz\in\X$, $n\in\N$ and
$f_j\in\PX{\xz}$ for $j\in\N$
with $1\leq j\leq n$.
Then the intersection
$\bigcap_{j=1}^n\MY{\Phi(f_j)}$ is a
non-empty set.
\end{lem}

\begin{proof}
Let $g=\sum_{j=1}^nf_j/n$
and fix an arbitrary $k\in\N$
with $1\leq k\leq n$.
Since $0\leq f_k\leq1$
and $f_k(\xz)=1$,
we have $0\leq g\leq1$ and
$g(\xz)=1$.
This yields $g\in\PX{\xz}$.
It is straightforward to verify that
\[
\MX{g}=\bigcap_{j=1}^n\MX{f_j}
\qqbox{and}
\ZX{g}=\bigcap_{j=1}^n\ZX{f_j}.
\]
In particular, we have $\MX{g}\subseteq\MX{f_k}$ and $\ZX{g}\subseteq\ZX{f_k}$.
We conclude that $\DX(f_k)\subseteq\DX(g)$ by Lemma~\ref{lem3.3}, and hence $\Phi(\DX(f_k))\subseteq\Phi(\DX(g))$.
Combining with Lemma~\ref{lem3.4}, this gives $\DY(\Phi(f_k))\subseteq\DY(\Phi(g))$.
It follows from Lemma~\ref{lem3.3} that
$\MY{\Phi(g)}\subseteq\MY{\Phi(f_k)}$.
Therefore $\MY{\Phi(g)}\subseteq\bigcap_{j=1}^n\MY{\Phi(f_j)}$
as $k\in\N$ with $1\leq k\leq n$
was arbitrary.
Moreover, since $\Phi(g)\in\SCY$, we have $\MY{\Phi(g)}\neq\emptyset$.
As a consequence, we obtain $\bigcap_{j=1}^n\MY{\Phi(f_j)}\neq\emptyset$.
This completes the proof.
\end{proof}

By extending the finite case established in
Lemma~\ref{lem3.6},
we obtain the following result
for arbitrary families of peak functions.

\begin{lem}\label{lem3.7}
The intersection $\bigcap_{f\in\PX{\xz}}\MY{\Phi(f)}$ is a non-empty set for each $\xz\in\X$. \end{lem}

\begin{proof}
Fix $\xz\in\X$ arbitrarily.
\begin{comment}
Then there exists an $f_0\in\PX{\xz}$.
Since $f_0\in\PX{\xz}$, we have the following equality:
\[
\bigcap_{f \in \PX{\xz}} \MY{\Phi(f)}
= \bigcap_{f \in \PX{\xz}}\Bigl(\MY{\Phi(f)} \cap \MY{\Phi(f_0)}\Bigr).
\]
\end{comment}
Because $\Phi(f)\in\SCY$ vanishes at infinity, the set $\MY{\Phi(f)}$ is a compact subset of $\Y$ for all $f\in\PX{\xz}$.
\begin{comment}
For each $f\in\PX{\xz}$,
the set
$\MY{\Phi(f)}\cap\MY{\Phi(f_0)}$
is a closed subset
of the compact set $\MY{\Phi(f_0)}$.
\end{comment}
By Lemma~\ref{lem3.6},
\begin{comment}
every finite family
$\set{f_1,\dots,f_n}\subset\PX{\xz}$
satisfies
\[
\bigcap_{j=1}^n\Big(\MY{\Phi(f_j)}\cap\MY{\Phi(f_0)}\Big)
=\bigcap_{k=0}^n\MY{\Phi(f_k)}
\neq\emptyset,
\]
so the collection 
$\set{\MY{\Phi(f)}\cap\MY{\Phi(f_0)}:f \in\PX{\xz}}$
\end{comment}
the collection $\{\MY{\Phi(f)} : f\in\PX{x_0}\}$ of compact subsets of $Y$
has the finite intersection property.
We thus conclude that
\begin{comment}
\[
\bigcap_{f\in\PX{\xz}}\Big(\MY{\Phi(f)}\cap\MY{\Phi(f_0)}\Big)
\neq\emptyset.
\]
This proves that
\end{comment}
\( \bigcap_{f \in \PX{\xz}} \MY{\Phi(f)}\)
is a non-empty set for each $\xz\in\X$.
\end{proof}

The following lemma identifies a unique point
$y\in Y$ associated with each $x\in X$,
which allows us to define a map
between $X$ and $Y$.

\begin{lem}\label{lem3.8}
For each $\xz\in\X$ there exists a unique
$\yz\in\Y$ such that
$\Phi(\PX{\xz})=\PY{\yz}$.
\end{lem}

\begin{proof}
Fix an arbitrary $\xz\in\X$.
By Lemma~\ref{lem3.7}, we have
$\bigcap_{f\in\PX{\xz}}\MY{\Phi(f)}\neq\emptyset$.
Choose
$\yz\in\bigcap_{f\in\PX{\xz}}\MY{\Phi(f)}$.
We claim that
$\Phi(\PX{\xz})\subseteq\PY{\yz}$.
Let $u\in\Phi(\PX{\xz})$ be fixed.
Then there exists an $h\in\PX{\xz}$
such that $u=\Phi(h)$.
By virtue of the property that
$\yz\in\bigcap_{f\in\PX{\xz}}\MY{\Phi(f)}$,
we have $\yz\in\MY{\Phi(h)}=\MY{u}$.
This implies that $u(\yz)=1$,
hence $u\in\PY{\yz}$.
This shows that
\[
\Phi(\PX{\xz})\subseteq\PY{\yz}.
\]
Applying the same argument to 
$(\Phii,\yz)$ yields
$\Phii(\PY{\yz})\subseteq\PX{\xo}$
for some $\xo\in\X$, that is,
\[
\PY{\yz}\subseteq\Phi(\PX{\xo}).
\]
Combining the previous inclusions gives
$\Phi(\PX{\xz})\subseteq
\PY{\yz}\subseteq\Phi(\PX{\xo})$.
Since $\Phi$ is injective,
it follows that
$\PX{\xz}\subseteq\PX{\xo}$.
By Lemma~\ref{lem3.2},
we have $\xz=\xo$.
It follows from the previous inclusions that
$\Phi(\PX{\xz})=\PY{\yz}$.

Finally, we prove the uniqueness of such
point $\yz\in\Y$.
Suppose that another point $\yo\in\Y$ satisfies
$\Phi(\PX{\xz})=\PY{\yo}$ as well.
Then we have
$\PY{\yz}=\Phi(\PX{\xz})=\PY{\yo}$.
By applying Lemma~\ref{lem3.2},
we obtain $\yz=\yo$,
which proves the uniqueness of the point $\yz\in\Y$
that satisfies $\Phi(\PX{\xz})=\PY{\yz}$.
\end{proof}

Finally, we record the following auxiliary
lemma, which provides a technical
construction used in the proof of
Theorem~\ref{thm3.1}.
The idea of proof originated in
\cite[Lemma~3.14]{miu}.

\begin{lem}\label{lem3.9}
Let $v\in\SCY$ and $\yz\in\Y$ be such that $v(\yz)<1$.
Then there exists a $u_0\in\PY{\yz}$
such that $v+(1-v(\yz))u_0\in\PY{\yz}$.
\end{lem}

\begin{proof}
We set $r=1-v(\yz)$.
Then $r>0$ by assumption.
We define
\begin{align*}
Y_0
&=
\set{y\in\Y:|v(y)-v(\yz)|\geq\fr{r}{4}},\\
Y_n
&=
\set{y\in\Y:
\fr{r}{2^{n+2}}\leq|v(y)-v(\yz)|
\leq\fr{r}{2^{n+1}}},\quad( n\in\mathbb{N}).
\end{align*}
We observe that $Y_k$ is a closed subset of $\Y$ with $\yz\not\in Y_k$ for all $k\in\N\cup\set{0}$.
For each $n\in\N$ there exists a $u_n\in\SCY$ such that
\begin{equation}\label{lem3.9.1}
u_n(\yz)=1
\qqbox{and}
u_n=0
\qbox{on $Y_0\cup Y_n$.}
\end{equation}
We define a function $u_0\colon\Y\to\R$
as follows:
\[
u_0=\sum_{n=1}^\infty\fr{u_n}{2^n}.
\]
Then $u_0$ converges
in $C_0(Y)$ with $u_0\geq0$, as $u_n\in\SCY$
for $n\in\N$.
Since $u_n(\yz)=1=\Vinf{u_n}$ for $n\in\mathbb{N}$, we obtain the following inequalities:
\[
1=u_0(\yz)
\leq\Vinf{u_0}
\leq\sum_{n=1}^\infty\fr{\Vinf{u_n}}{2^n}=1.
\]
Hence, $u_0(\yz)=1=\Vinf{u_0}$.
This yields $u_0\in\PY{\yz}$.

We define a function
$w\colon\Y\to\R$ as follows:
\[
w=v+ru_0.
\]
We obtain $w(\yz)=1$
as $r=1-v(\yz)$ and $u_0(\yz)=1$.
We prove that $w\in\SCY$, that is, $0\leq w(y)\leq1$ for all $y\in\Y$.
Fix an arbitrary $y\in\Y$.
We have $w(y)\geq0$ as $v,u_0\geq0$
with $r>0$.
We will consider three cases
to show that $w(y)\leq1$.
\begin{itemize}
\item
If $y\in Y_0$, then equality \eqref{lem3.9.1}
shows that $u_n(y)=0$ for all $n\in\N$,
hence $u_0(y)=0$.
Then we obtain
$w(y)=v(y)\leq1$.

\item
If $y\in Y_m$ for some $m\in\N$,
then we have
$v(y)-v(\yz)\leq r/2^{m+1}$
by definition.
By combining the last inequality
with $r=1-v(\yz)$,
we obtain the following result:
\[
v(y)
\leq
v(\yz)+\fr{r}{2^{m+1}}
=1+r\left(\fr{1}{2^{m+1}}-1\right).
\]
We derive from equality \eqref{lem3.9.1}
with $y\in Y_m$
that $u_m(y)=0$.
We obtain
\[
u_0(y)
=\sum_{n\neq m}^\infty\fr{u_n(y)}{2^n}
\leq1-\fr{1}{2^m}.
\]
By combining the preceding inequalities,
we obtain the following results:
\begin{align*}
w(y)
&=
v(y)+ru_0(y)\\
&\leq
1+r\left(\fr{1}{2^{m+1}}-1\right)
+r\left(1-\fr{1}{2^m}\right)\\
&=
1-\fr{r}{2^{m+1}}<1.
\end{align*}
Therefore, we have $w(y)\leq1$.

\item
If $y\not\in\cup_{k=0}^\infty Y_k$,
then we obtain $v(y)=v(\yz)=1-r$.
We have the following inequality:
\[
w(y)=1-r+ru_0(y)
\leq1.
\]
We conclude that $w(y)\leq1$.
\end{itemize}
As a consequence, we conclude that
$w(y)\leq1$ for all $y\in\Y$.
This shows that $w\in\PY{\yz}$.
This completes the proof.
\end{proof}

The preceding lemmas provide
all the necessary ingredients,
and we can now assemble them
to prove the main result.

\begin{proof}[\textbf{Proof of Theorem~\ref{thm3.1}}]
We first introduce two maps,
$\tau,\si$, between $\X$ and $\Y$.
By Lemma~\ref{lem3.8}, there exists a map
$\tau\colon\X\to\Y$ that satisfies
$\Phi(\PX{x})=\PY{\tau(x)}$
for all $x\in\X$.
Similarly, applying Lemma~\ref{lem3.8} to $\Phii$
instead of $\Phi$, we obtain
a map $\si\colon\Y\to\X$ satisfying that
$\Phii(\PY{y})=\PX{\si(y)}$
for all $y\in\Y$.
Consequently, this yields the following identities:
\begin{align*}
&\Phi(\PX{x})=\PY{\tau(x)} \hspace{0.95cm}(x\in X),\\[2pt]
&\Phii(\PY{y})=\PX{\si(y)} \hspace{0.6cm}(y\in Y).
\end{align*}
By combining the
last two identities,
we obtain
\begin{equation}\label{thm3.1.1}
\begin{aligned}
\PX{x}&=\Phii(\PY{\tau(x)})=\PX{\si(\tau(x))} \hspace{0.6cm}(x\in X),\\[4pt]
\PY{y}&=\Phi(\PX{\si(y))}=\PY{\tau(\si(y))}\hspace{0.95cm}(y\in Y).
\end{aligned}
\end{equation}
It follows from Lemma~\ref{lem3.2} that
$x=\si(\tau(x))$ and $y=\tau(\si(y))$
for all $x\in\X$ and $y\in\Y$.
Hence, $\tau$ and $\si$ are bijections,
and moreover $\tau=\sii$.

Fix $f\in\SCX$ and $\yz\in\Y$
arbitrarily. We prove that 
$\Phi(f)(\yz)\leq f(\si(\yz))$.
To see this, we consider two cases.
First, suppose that $\Phi(f)(\yz)=1$.
By equality~\eqref{thm3.1.1}, we have
$\Phi(f)\in\PY{\yz}=\Phi(\PX{\si(\yz)})$,
which implies $f\in\PX{\si(\yz)}$.
This shows that
$f(\si(\yz))=1=\Phi(f)(\yz)$, and therefore the desired inequality holds in this case.
We next consider the case where
$\Phi(f)(\yz)<1$.
Let $u_0\in\PY{\yz}$ be as in Lemma~\ref{lem3.9}, corresponding to $v=\Phi(f)$.
By setting $r=1-\Phi(f)(\yz)$,
Lemma~\ref{lem3.9} shows
that $u_0$ satisfies
$\Phi(f)+ru_0\in\PY{\yz}$.
For brevity in notation,
we will write $w=\Phi(f)+ru_0$.
As $w\in\PY{\yz}$,
it follows from equality \eqref{thm3.1.1} that
$w\in\Phi(\PX{\si(\yz)})$,
hence $\Phii(w)\in\PX{\si(\yz)}$.
In particular, $\Phii(w)(\si(\yz))=1$.
Because $\Phi$ is an isometry,
we have:
\begin{align*}
1-f(\si(\yz))
&\leq
|\Phii(w)(\si(\yz))-f(\si(\yz))|\\
&\leq
\Vinf{\Phii(w)-f}\\
&=
\Vinf{w-\Phi(f)}\\
&=\Vinf{ru_0}
=
1-\Phi(f)(\yz),
\end{align*}
where we used $\Vinf{u_0}=1$.
From this we conclude that $\Phi(f)(\yz)\leq f(\si(\yz))$ also holds if $\Phi(f)(\yz)<1$.
Therefore, $\Phi(f)(y)\leq f(\si(y))$ for all $f\in\SCX$ and $y\in\Y$.

Applying the above arguments
to the pair of $(\Phii,\sii)$,
instead of $(\Phi,\si)$, we obtain
$\Phii(u)(x)\leq u(\sii(x))$
for all $u\in\SCY$ and $x\in\X$.
By letting $u=\Phi(f)$ and $x=\si(y)$
for $f\in\SCX$ and $y\in\Y$
in the last inequality,
we conclude that $f(\si(y))\leq\Phi(f)(y)$.
As a consequence, the following equality holds
for all $f\in\SCX$ and $y\in\Y$:
\begin{equation}\label{thm3.1.2}
\Phi(f)(y)=f(\si(y)).
\end{equation}

Now, we prove that $\si\colon\Y\to\X$
is continuous.
Fix an arbitrary open subset
$U\subset\X$.
We need to show that $\sii(U)$ is an
open subset of $\Y$.
Let $\yo\in\sii(U)$ be an arbitrary point.
There exists an $f_1\in\SCX$ such that
$f_1(\si(\yo))=1$ and $f_1=0$
on $\X\setminus U$.
Define a subset $V$
of $\Y$ as follows:
\[
V=\set{y\in\Y:\Phi(f_1)(y)>1/2}.
\]
It is immediate that $V$ is an open subset
of $\Y$.
It follows from
equality
\eqref{thm3.1.2}
that 
$\Phi(f_1)(\yo)=f_1(\si(\yo))=1$,
which yields that $\yo\in V$.
For each $y\in V$, we obtain
$f_1(\si(y))=\Phi(f_1)(y)>1/2$
by equality
\eqref{thm3.1.2}.
This implies that $\si(V)\subseteq U$
because $f_1=0$ on $\X\setminus U$.
We conclude that $\yo\in V\subseteq\sii(U)$ with $V$ open,
which proves that $\sii(U)$ is an open subset
of $\Y$.
As a consequence, $\si$ is continuous on $\Y$.

The above argument, applied to
the pair of $(\Phii,\sii)$, instead of
$(\Phi,\si)$, shows that
$\sii\colon\X\to\Y$ is continuous.
We conclude that
$\si\colon\Y\to\X$ is a homeomorphism.
This completes the proof of Theorem~\ref{thm3.1}.
\end{proof}

\begin{proof}[\textbf{Proof of Corollary~\ref{cor1}}]
Let $\Phi\colon\SCX\to\SCY$ be
a surjective isometry.
Then there exists a homeomorphism
$\si\colon\Y\to\X$ such that
$\Phi(f)=f\circ\si$ for all $f\in\SCX$ by
Theorem~\ref{thm3.1}.
Now we define a map
$\Phit\colon\CX\to\CY$ as follows:
\[
\Phit(f)=f\circ\si
\qq(f\in\CX).
\]
It is immediate that $\Phit$ is a real-linear isometry
that satisfies $\Phit=\Phi$ on $\SCX$.
It is plain that the map $g\mapsto g\circ\sigma^{-1}$ from $\CY$ to $\CX$ is the inverse of $\Phit$.
In particular, $\Phit$ is a surjective isometry from $\CX$ onto $\CY$.

Finally we show that the extension
of $\Phi$ is unique.
Let $\Phit'\colon\CX\to\CY$ be
another surjective real-linear isometry
with $\Phit'=\Phi$ on $\SCX$.
By the Banach--Stone theorem,
there exists a continuous function
$\alpha\colon\Y\to\set{\pm1}$
and a homeomorphism
$\varphi\colon\Y\to\X$ such that
$\Phit'(f)=\alpha(f\circ\varphi)$
for all $f\in\CX$.
For each $y\in\Y$, there exists
an $f_y\in\SCX$ that satisfies
$f_y(\varphi(y))=1$.
We obtain
\[
\alpha(y)
=\alpha(y)f_y(\varphi(y))
=\Phit'(f_y)(y)
=\Phi(f_y)(y)
\geq0.
\]
Because $\alpha(y)\in\set{\pm1}$,
we obtain $\alpha(y)=1$
for all $y\in\Y$.
As a consequence, we obtain
$f\circ\varphi=\Phit'(f)
=\Phit(f)=f\circ\si$
for all $f\in\SCX$.
Since $\SCX$ separates the points
of $\X$,
we observe that $\varphi=\si$,
which shows that $\Phit'=\Phit$.
This completes the proof.
\end{proof}

\section{Applications to phase-isometries}
\label{sect4}

In this section we apply our main theorem
to the study of phase-isometries
on the positive unit sphere.
Let $A,B$ be subsets of Banach spaces.
A map $T\colon A\to B$ is called
a \textit{phase-isometry} if it satisfies
the following equality for all $a,b\in A$:
\[
\set{\|T(a)+T(b)\|,\|T(a)-T(b)\|}
=\set{\|a+b\|,\|a-b\|}.
\]
Phase-isometries have been extensively
investigated in the literature
(see, for example, \cite{ilis,sun}).
Our main result immediately yields
the following theorem,
which provides a concise characterization
of phase-isometries,
and generalizes \cite[Theorem~1.1]{hiro1}.

\begin{thm}\label{thm2}
Let $T\colon\SCX\to\SCY$ be
a surjective phase-isometry.
Then there exists a homeomorphism
$\si\colon\Y\to\X$ such that
$T(f)=f\circ\si$ holds
for all $f\in\SCX$.
Consequently, $T$ extends uniquely
to a surjective real-linear isometry
from $\CX$ onto $\CY$.
\end{thm}

\begin{proof}
Let $T\colon\SCX\to\SCY$ be a surjective
phase-isometry.
We prove that $T$ is an isometry.
Although this is essentially proved in
\cite[Proposition~2.1]{hiro1},
we give its proof
for the sake of completeness.
Let $f,g\in\SCX$ and $x\in\X$ be fixed.
Because $f(x),g(x)\geq0$,
we obtain the following inequalities:
\[
|f(x)-g(x)|
\leq
|f(x)+g(x)|
\leq
\Vinf{f+g},
\]
which yields $\Vinf{f-g}\leq\Vinf{f+g}$.
By applying the same argument
to $T(f),T(g)$, we conclude that
$\Vinf{T(f)-T(g)}\leq\Vinf{T(f)+T(g)}$.
As $T$ is a phase-isometry,
the preceding inequalities show
the following result:
\begin{align*}
\Vinf{T(f)-T(g)}
&=
\min\set{\Vinf{T(f)+T(g)},\Vinf{T(f)-T(g)}}\\
&=
\min\set{\Vinf{f+g},\Vinf{f-g}}\\
&=
\Vinf{f-g}.
\end{align*}
Therefore we conclude that $T$ is
a surjective isometry.
Now we apply Theorem~\ref{thm3.1} to $T$.
Then we conclude that $T$ is a
composition operator induced by
a homeomorphism.
By Corollary~\ref{cor1},
$T$ extends uniquely to a surjective
real-linear isometry between
$\CX$ and $\CY$.
\end{proof}

\begin{remark}
In \cite{leun}, Leung, Ng and Wong present an elegant argument
for the case $S(C(K))^+$, the positive unit sphere
of $C(K)$ with $K$ compact Hausdorff space.
The authors analyze strictly positive functions
in $S(C(K))^+$ and their maximum sets.
Because $K$ is compact, there exists the constant function
$\mathds{1}\in S(C(K))^+$.
Then the maximum set,
$\MK{f}$, of $f\in S(C(K))^+$ is
that of a strictly positive function $(\mathds{1}+f)/2$.
Moreover, the zero set,
$\ZK{f}$, of $f\in S(C(K))^+$ with $\ZK{f}\neq\emptyset$ is related
to a maximum set of a strictly positive function $\mathds{1} - f/2$,
that is $\ZK{f}=\MK{\mathds{1}-f/2}$.
In this sense, strictly positive functions
and the constant function $\mathds{1}$
play essential role in their argument.
In contrast,
in the case of $\SCX$ with $X$ locally compact Hausdorff space,
it is known that the existence of strictly positive functions
in $\CX$ is equivalent to
$\X$ being $\sigma$-compact~\cite[pp. 187, Exercise~4]{Bauer01}.
In particular, the lack of assumption that
$\X$ is $\sigma$-compact obstructs
an adaptation of the compact-case argument in \cite{leun}.

The key idea of this paper is to consider all functions
in $\SCX$ along with their zero sets and
maximum sets, rather than restricting attention
to strictly positive functions with maximum sets.
In Lemma~\ref{lem3.3}, we obtain some properties
of the sets $\DX(f)$ for $f\in\SCX$
with respect to the maximum sets and zero sets.
By Lemma~\ref{lem3.4}, these properties are preserved
by surjective isometry $\Phi$.
This allows us to define a map $\si$ from $\Y$ to $\X$.
Then we can prove that $\Phi$ is a composition operator
induced by $\si$.
\end{remark}

\section*{Acknowledgment}
The second author
 is supported by the Ministry of Education, Culture, Sports, Science and Technology of Japan (MEXT), and by the Ministry of Science and Technology (MoST) of Taiwan under Grant No. 115-2917-I-110-009.
The third author partially supported by JSPS KAKENHI Grant Number JP 25K07028.
The authors are grateful to Professor Ngai-Ching Wong, a co-advisor of the second author, for his valuable suggestions and comments.


\begin{thebibliography}{99}
\bibitem{Bauer01}
H.~Bauer,
Measure and integration theory.
Transl. from the German by Robert B. Burckel.
de Gruyter Studies in Mathematics. 26.
Berlin: de Gruyter. xvi, 230 p. (2001).

\bibitem{cue1}
M.~Cueto-Avellaneda,
D.~Hirota,
T.~Miura,
A.M.~Peralta,
Exploring new solutions to Tingley's problem
for function algebras,
Quaest. Math.
46, No. 7, 1315--1346 (2023).

\bibitem{cue2}
M.~Cueto-Avellaneda,
A.M.~Peralta,
On the Mazur--Ulam property
for the space of Hilbert-space-valued
continuous functions,
J. Math. Anal. Appl.
479, No. 1, 875--902 (2019).

\bibitem{hat1}
O.~Hatori,
The Mazur--Ulam property for uniform algebras,
Stud. Math.
265, No.2, 227--239 (2022).

\bibitem{hat2}
O.~Hatori,
The Mazur--Ulam property
for a Banach space which satisfies
a separation condition,
RIMS K\^{o}ky\^{u}roku Bessatsu
B93, 29--82 (2023).

\bibitem{hat3}
O.~Hatori,
S.~Oi,
R.~Shindo Togashi,
Tingley's problems on uniform algebras,
J. Math. Anal. Appl.
503, No.2, Article ID 125346, 14 p. (2021).

\bibitem{hiro1}
D.~Hirota,
I.~Matsuzaki,
T.~Miura,
Phase-isometries between the
positive cones of the Banach space
of continuous real-valued functions,
Ann. Funct. Anal.
15, No. 4, Paper No. 77, 11 p. (2024).

\bibitem{hiro2}
D.~Hirota,
T.~Miura,
Tingley's problem for a Banach space
of Lipschitz functions
on the closed unit interval,
RIMS K\^{o}ky\^{u}roku Bessatsu
B93, 157--181 (2023).

\bibitem{ilis}
D.~Ili\v{s}evi\'{c},
M.~Omladi\v{c},
A.~Turn\v{s}ek,
Phase-isometries between normed spaces,
Linear Algebra Appl.
612, 99--111 (2021).

\bibitem{leun}
C.W.~Leung,
C.K.~Ng,
N.C.~Wong,
On a variant of Tingley's problem
for some function spaces,
J. Math. Anal. Appl.
496, No. 1, Article ID 124800, 17 p. (2021).

\bibitem{miu}
T.~Miura,
Real-linear isometries between function algebras,
Cent. Eur. J. Math.
9, No. 4, 778--788 (2011).

\bibitem{pera}
A.M.~Peralta,
A survey on Tingley’s problem for operator algebras,
Acta Sci. Math. 84, No. 1--2, 81--123 (2018).

\bibitem{sun}
L.~Sun, Y.~Sun, D.~Dai,
On phase-isometries between the positive cones
of continuous function spaces,
Ann. Funct. Anal.
14, No. 1, Paper No. 17, 12 p. (2023).

%\bibitem{tan1}
%D.~Tan,
%Y.~Gao,
%Phase-isometries on the unit sphere of $C(K)$,
%Ann. Funct. Anal.
%12, No. 1, Paper No. 15, 14 p. (2021).
%
%\bibitem{tan3}
%D.~Tan,
%F.~Zhang,
%X.~Huang,
%Phase-isometries on the unit sphere of CL-spaces,
%J. Math. Anal. Appl.
%527, No. 2, Article ID 127568, 9 p. (2023).

\bibitem{ting}
D.~Tingley,
Isometries of the unit sphere,
Geom. Dedicata,
22, No. 3, 371--378 (1987).

\bibitem{wang1}
R.~Wang,
Isometries between the unit spheres
of $C_0(\Omega)$ type spaces,
Acta Math. Sci.
14, No. 1, 82--89 (1994).

\bibitem{wang2}
R.~Wang,
Isometries of $C_0^{(n)}(X)$,
Hokkaido Math. J.
25, 465--519 (1996).

\bibitem{wang4}
R.~Wang,
A.~Orihara,
Isometries on the $\ell^1$-sum
of $C_0(\Omega,E)$ type spaces,
J. Math. Sci., Tokyo
2, No. 1, 131--154 (1995).
\end{thebibliography}
\end{document}